\documentclass{article}
\usepackage[utf8]{inputenc}
\usepackage[all,cmtip]{xy}
\usepackage{amssymb,amsmath,amsthm,MnSymbol}
\usepackage{mathtools,mathdots}
\usepackage{mathrsfs,calligra, bbm}
\usepackage{pifont}
\usepackage[top=1.4in, bottom=1.4in, left=1in, right=1in]{geometry}
\usepackage{graphicx}
\usepackage{CJKutf8}
\usepackage{dirtytalk}
\usepackage{hyperref}
\usepackage{tikz-cd, tikz}
\usepackage{csquotes}
\usepackage{enumitem}
\usepackage[T1]{fontenc}
\usepackage{titlesec}
\DeclareMathAlphabet{\mathpzc}{OT1}{pzc}{m}{it}
\usepackage[british]{babel}
\usepackage[backend=biber,style=alphabetic,sorting=anyt, maxbibnames=99,maxcitenames=99,maxalphanames=99,
]{biblatex}
\usepackage{lipsum}
\usepackage{bookmark}
\bookmarksetup{
  numbered,
  open
}

\DeclareMathAlphabet{\mathpzc}{OT1}{pzc}{m}{it}

\newtheorem{Definition}{Definition}[section]
\newtheorem{Theorem}[Definition]{Theorem}
\newtheorem{Lemma}[Definition]{Lemma}

\theoremstyle{remark}
\newtheorem{Remark}[Definition]{Remark}

\DeclareMathOperator\B{\mathbf{B}}
\DeclareMathOperator\C{\mathbf{C}}
\DeclareMathOperator\D{\mathbf{D}}

\DeclareMathOperator\Q{\mathbf{Q}}
\DeclareMathOperator\R{\mathbf{R}}

\DeclareMathOperator\Z{\mathbf{Z}}

\DeclareMathOperator\calD{\mathcal{D}}

\DeclareMathOperator\calL{\mathcal{L}}

\DeclareMathOperator\calO{\mathcal{O}}

\DeclareMathOperator\calR{\mathcal{R}}

\DeclareMathOperator\frakp{\mathfrak{p}}

\DeclareMathOperator\GL{GL}

\DeclareMathOperator\Hom{Hom}

\DeclareMathOperator\cl{cl}

\DeclareMathOperator\coker{coker}

\DeclareMathOperator\cris{cris}

\DeclareMathOperator\cyc{cyc}

\DeclareMathOperator\dR{dR}

\DeclareMathOperator\Fil{Fil}

\DeclareMathOperator\FM{FM}

\DeclareMathOperator\Gal{Gal}
\DeclareMathOperator\GB{GB}

\DeclareMathOperator\Gr{Gr}

\DeclareMathOperator\id{id}
\DeclareMathOperator\Ind{Ind}

\DeclareMathOperator\image{image}

\DeclareMathOperator\rank{rank}

\DeclareMathOperator\rig{rig}

\DeclareMathOperator\st{st}

\DeclareMathOperator\Mod{\mathbf{\mathsf{Mod}}}

\DeclareMathOperator\Rep{\mathbf{\mathsf{Rep}}}

\newcommand\scalemath[2]{\scalebox{#1}{\mbox{\ensuremath{\displaystyle #2}}}}

\makeatletter
\renewcommand{\maketitle}{\bgroup\setlength{\parindent}{0pt}
\begin{flushleft}
  \LARGE{\textbf{\@title}}
  
  \vspace{4mm}
  
  \large{\textsc{\@author}} \hfill \normalfont{\text{\@date}}
  
  \vspace{4mm}
\end{flushleft}\egroup
}
\makeatother

\title{On Greenberg--Benois $\mathcal{L}$-invariants and Fontaine--Mazur $\mathcal{L}$-invariants}
\author{Ju-Feng Wu}
\date{}

\begin{document}

\maketitle

{\footnotesize \noindent \textbf{Abstract.} We prove a comparison theorem between Greenberg--Benois $\calL$-invariants and Fontaine--Mazur $\calL$-invariants. Such a comparison theorem supplies an affirmative answer to a speculation of Besser--de Shalit.}

\tableofcontents

\section{Introduction}\label{section: intro}

Let $f$ be a cuspidal normalised newform of weight $2k$ and level $\Gamma_0(pN)$, where $p$ is a prime number and $N\in \Z_{>0}$ such that $p\nmid N$. Consider the complex $L$-function attached to $f$ \[
    L(f, s) = \sum_{n=1}^{\infty} a_n n^{-s}. 
\] 
One way to study $L(f, s)$ is via $p$-adic method. That is, one can associate $f$ with a \emph{$p$-adic $L$-function} $L_p(f, s)$, which $p$-adically interpolate the algebraic part of the special values $L(f, j)$ for $1\leq j\leq 2k-1$. In particular, the interpolation property at $s= k$ is given by the formula \[
    L_p(f, k) = \left( 1- \frac{p^{k-1}}{a_p}\right)\frac{L(f, k)}{\Omega_f},
\] where $\Omega_f$ is the Deligne period of $f$ at $k$ (\cite{Deligne}). 

Suppose moreover that $a_p = p^{k-1}$, the formula above shows that $L_p(f, s)$ vanishes at $s=k$. In the case when $k=1$, Mazur--Tate--Teitelbaum conjectured in \cite{MTT} that there exists an invariant $\calL(f)$ such that \[
    \frac{d}{ds} L_p(f, s)|_{s=k} = \calL(f) \frac{L(f, k)}{\Omega_f}.
\] This conjecture is known as the \emph{trivial zero conjecture} and has been proven by Greenberg--Stevens in \cite{Greenberg--Stevens}. Moreover, for higher weights, various generalisations of the invariant $\calL(f)$ has been proposed. The following is an incomplete list: \begin{enumerate}
    \item[$\bullet$] In \cite{Greenberg}, R. Greenberg constructed the $\calL$-invariants for Galois representations that are ordinary at $p$ and suggested a generalisation of the trivial zero conjecture. 
    \item[$\bullet$] In \cite{Mazur}, Fontaine--Mazur defined the $\calL$-invariant by studying the semistable module (à la Fontaine) associated with a $p$-adic representation. 
    \item[$\bullet$] In \cite{Coleman-padicShimura}, R. Coleman proposed a construction of $\calL$-invariants as an application of his $p$-adic integration theory. 
    \item[$\bullet$] In \cite{Teitelbaum}, J. Teitelbaum proposed a construction of $\calL$-invariants by applying the $p$-adic integration theory to $p$-adically uniformised Shimura curve. 
\end{enumerate}
All these $\calL$-invariants are known to be equal: Coleman--Iovita compared the second and the third in \cite{Coleman--Iovita}; Iovita--Spieß compared the second and the fourth in \cite{Iovita-Spiess}; and the comparison between the first and the second is a special case of \cite[Proposition 2.3.7]{Benois}.

It is a natural question to ask whether one can establish a similar philosophy for higher rank automorphic forms. Let us mention the following generalisations in our consideration: \begin{enumerate}
    \item[$\bullet$] In \cite{Benois}, D. Benois generalised Greeberg's construction to Galois representations of $\Gal_{\Q}$ that satisfies some reasonable conditions. He also stated a trivial zero conjecture in such a generality ([\emph{op. cit.}, pp. 1579]).
    \item[$\bullet$] In \cite{Besser--de_Shalit}, Besser--de Shalit generalised both the Fontaine--Mazur $\calL$-invariants and Coleman (or Teitelbaum) $\calL$-invariants by studying the $p$-adic cohomology groups of $p$-adically uniformised Shimura varieties. It is conjectured in \emph{loc. cit.} that these two constructions give rise to the same $\calL$-invariants (or $\calL$-operators as called in \emph{loc. cit.}). Authors of \emph{loc. cit.} also speculated that the existence of a \emph{trivial zero conjecture} for these two $\calL$-invariants. However, they were not able to provide an explicit statement.  
\end{enumerate}

This article concerns the comparison between Benois's $\calL$-invariants and the Fontaine--Mazur type $\calL$-invariants of Besser--de Shalit. To explain our result, let us fix some notations: Let $F$ be a number field such that for every prime ideal $\frakp\subset \calO_F$ sitting above $p$, the maximal unramified extension of $\Q_p$ in $F_{\frakp}$ is $\Q_p$ itself; let $E$ be a large enough value field that is a finite extension over $\Q_p$. Suppose $$
    \rho : \Gal_F \rightarrow \GL_n(E)
$$ is a Galois representation that is semistable at places above $p$. We further assume that $\rho$ satisfies the assumptions in \S \ref{subsection: assumptions}. In particular, we assume the Frobenius eigenvalues on the associated semistable modules are given by $p^m, ..., p^{m-n+1}$ (for some suitable $m\in \Z$ that is independent of the prime ideals sitting above $p$) and the monodromy is maximal. We remark in the beginning that these assumptions are required so that we can perform the following two constructions: \begin{enumerate}
    \item[$\bullet$] Following the suggestion in \cite{Rosso} (see also \cite{Hida-HilberLinvariant}), one can consider the induction $\Ind^{\Q}_F \rho$. Part of the assumptions then allows us to attach the $\calL$-invariant in Benois's style to $\Ind_{F}^{\Q}\rho(m)$. This resulting $\calL$-invariant is denoted by $\calL_{\GB}(\rho(m))$, where the subscript $\GB$ stands for `Greenberg--Benois'. We refer the readers to \S \ref{section: GB} for the construction of $\calL_{\GB}$. 
    \item[$\bullet$] We realised that the generalisation of Fontaine--Mazur $\calL$-invariants suggested by Besser--de Shalit can be translated to the world of semistable modules of a local Galois representation. The other part of the assumptions in \S \ref{subsection: assumptions} then allow us to attach $\calL$-invariants of Fontaine--Mazur style to each local Galois representation $\rho_{\frakp} = \rho|_{\Gal_{F_{\frakp}}}$ for every prime $\frakp$ above $p$. We term such $\calL$-invariants $\calL_{\FM}(\rho_{\frakp})$, where the subscript $\FM$ stands for `Fontaine--Mazur'. We refer the readers to \S \ref{section: FM} for the construction of $\calL_{\FM}$.
\end{enumerate}

Our main result reads as follows. 
\begin{Theorem}[Theorem \ref{Theorem: main; comparison}]
    We have an equality \[
        \calL_{\GB}(\rho(m)) = \prod_{\frakp | p} - \calL_{\FM}(\rho_{\frakp}),
    \]
    where the index set runs through all prime ideals in $\calO_F$ sitting above $p$
\end{Theorem}

Since there is a well-stated trivial zero conjecture for $\calL_{\GB}(\rho(m))$ in \cite{Benois}, our result immediately supplies an affirmative answer to Besser--de Shalit's speculation of the relationship between their $\calL$-invariants and $p$-adic $L$-functions. 

To close this introduction, let us mention that the generalisation of $\calL$-invariants à la Coleman (or Teitelbaum) suggested by Besser--de Shalit replaces Coleman's integration theory with Besser's theory of finite polynomial cohomology. Although they only consider the case for the trivial coefficient (so that we can only see automorphic forms of weight associated with the differential 1-forms), one can hope a generalisation for non-trivial coefficients by using \emph{finite polynomial cohomology with coefficients} (\cite{HW-fpcohomology}). We wish to come back to this in future projects and hopefully to compare this type of $\calL$-invariants with $\calL_{\FM}(\rho_{\frakp})$ as suggested in \cite{Besser--de_Shalit}.

\subsection*{Acknowledgements} 
This paper grew from a working group with Ting-Han Huang, Martí Roset Julià, and Giovanni Rosso and I would like to thank them for interesting discussions. I especially thank Givanni Rosso for valuable feedback regarding the early draft of this paper. I also thank David Loeffler for pointing out a mistake in an early version of this paper. I would also like to thank Muhammad Manji for interesting discussions on $(\varphi, \Gamma)$-modules. Part of the work was done when I was visiting National Center for Theoretical Sciences in Taipei; I would like to thank the hospitality of the institute and Ming-Lun Hsieh. Finally, I thank the anonymous referees for their corrections and valuable suggestions, which helped to improve the exposition of the article. This work is supported by the ERC Consolidator grant `Shimura varieties and the BSD conjecture' and the Irish Research Council under grant number IRCLA/2023/849 (HighCritical).

\subsection*{Notations}
\begin{enumerate}
    \item[$\bullet$] Through out this article, we fix a prime number $p$. 
    \item[$\bullet$] Given a field $F$, we fix a separable closure $\overline{F}$ and denote by $\Gal_F = \Gal(\overline{F}/F)$ its absolute Galois group. 
\end{enumerate}

\section{Preliminaries on \texorpdfstring{$(\varphi, \Gamma)$}{(phi, Gamma)}-modules}\label{section: (phi, Gamma)}

\subsection{General \texorpdfstring{$(\varphi, \Gamma)$}{(phi, Gamma)}-modules}\label{subsection: (phi, Gamma)}
Fix a compatible system of primitive $p$-power roots of unity $(\zeta_{p^n})_{n\in \Z_{\geq 0}}$ in $\overline{\Q}_p$. Given a finite extension $K$ of $\Q_p$, consider $K(\zeta_{p^{\infty}}) = \bigcup_{n\in \Z_{\geq 0}} K(\zeta_{p^n})$ and denote by $\Gamma = \Gamma_K$ the Galois group $\Gal(K(\zeta_{p^{\infty}})/K)$. Moreover, for any $r\in [0, 1)$, let \[
    \calR_K^r := \left\{ f = \sum_{i\in \Z}a_i T^i: \begin{array}{l}
        a_i \in K^{\mathrm{unr}} \cap K(\zeta_{p^{\infty}}) \\
        f\text{ is holomorphic on the annulus } r\leq |T| <1
    \end{array}\right\}
\] and \[
    \calR_K := \bigcup_{r\in [0, 1)} \calR_K^r,
\]
where $K^{\mathrm{unr}}$ is the maximal unramified extension of $K$ in $\overline{\Q}_p$ and the infinite union is taken with respect to the inclusions $\calR_K^{r} \hookrightarrow \calR_K^{r'}$ for $r\leq r'<1$. We call the ring $\calR_K$ the \emph{Robba ring} over $K$. It carries a $\varphi$-action and a $\Gamma$-action via the formula \[
    \varphi(T) = (1+T)^p -1 \quad \text{ and }\quad \gamma(T) = (1+T)^{\chi_{\cyc}(\gamma)}-1 \text{ for any }\gamma \in \Gamma,
\]
where $\chi_{\cyc}$ is the cyclotomic character.

In what follows, we shall consider a more generalised version of $\calR_K$. Let $E$ be a finite extension of $\Q_p$, we denote by $\calR_{K, E} := \calR_K \otimes_{\Q_p}E$ and call it the \emph{Robba ring over $K$ with coefficients in $E$}. We linearise the actions of $\varphi$ and $\Gamma$ on $\calR_{K, E}$ via $\varphi \otimes \id$ and $\gamma \otimes \id$ respectively. In what follows, we often assume $E$ is large enough so that $K \subset E$. 

By a \emph{$(\varphi, \Gamma)$-module} over $\calR_{K, E}$, we mean a finite free $\calR_{K, E}$-module $D$ together with a $\varphi$-semilinear endomorphism $\varphi_D$ and a semilinear action by $\Gamma$, which commute with each other, such that the induced map \[
    \varphi_D: \varphi^*D = D \otimes_{\varphi}\calR_{K, E} \rightarrow D
\] is an isomorphism. We shall denote by $\Mod_{\calR_{K, E}}^{(\varphi, \Gamma)}$ the category of $(\varphi, \Gamma)$-modules over $\calR_{K, E}$.

Let $\Rep_{K}(E)$ the category of Galois representations of $\Gal_K$ with coefficients in $E$. Then, by \cite[Proposition 1.1.4]{Benois}, there is a fully faithful functor \[
    \D_{\rig}^{\dagger} : \Rep_{K}(E) \rightarrow \Mod_{\calR_{K, E}}^{(\varphi, \Gamma)}.
\] 
Moreover, by letting $\Mod_{K, E}^{(\varphi, N)}$ (resp., $\Mod_{K, E}^{\varphi}$) the category of $(\varphi, N)$-modules (resp., $\varphi$-modules) over $K_0 = K \cap \Q_p^{\mathrm{unr}}$ with coefficients in $E$, there is a functor (see, for example, \cite[\S 1.2.3]{Benois}) \[
    \calD_{\st}: \Mod_{\calR_{K, E}}^{(\varphi, \Gamma)} \rightarrow \Mod_{K, E}^{(\varphi, N)} \quad (\text{resp., }\calD_{\cris}: \Mod_{\calR_{K, E}}^{(\varphi, \Gamma)} \rightarrow \Mod_{K, E}^{\varphi})
\] such that if $\rho \in \Rep_{K}(E)$ is semistable (resp., crystalline), then (\cite[Théorème 0.2]{Berger-differential}) \[
    \calD_{\st}(\D_{\rig}^{\dagger}(\rho)) = \D_{\st}(\rho) \quad (\text{resp., } \calD_{\cris}(\D_{\rig}^{\dagger}(\rho)) = \D_{\cris}(\rho)).
\] 
Here $\D_{\st}$ (resp., $\D_{\cris}$) is Fontaine's semistable (resp., crystalline) functor (\cite{Fontaine-period, Berger-differential}), assigning a Galois representation in $\Rep_{K}(E)$ a $(\varphi, N)$-module (resp., $\varphi$-module) over $K_0$ with coefficients in $E$.

Now, let $D$ be a $(\varphi, \Gamma)$-module over $\calR_{K, E}$. Recall the cohomology of $D$ is defined by the cohomology of the \emph{Herr complex} \[
    0 \rightarrow D \xrightarrow{x \mapsto ((\varphi_D-1)x, (\gamma -1)x)} D \oplus D \xrightarrow{(x,y) \mapsto (\gamma -1)x - (\varphi_D -1)y} D \rightarrow 0,
\]
where $\gamma$ is a (fixed) topological generator of $\Gamma$. Note that, given $\alpha = (x, y)\in D \oplus D$ such that  $(\gamma -1)x -(\varphi_D -1)y = 0$, there is an extension \[
    0 \rightarrow D \rightarrow D_{\alpha} \rightarrow \calR_{K, E} \rightarrow 0
\] defined by \begin{equation}\label{eq: construction of an extension}
    D_{\alpha} = D \oplus \calR_{K, E}e, \quad (\varphi_{D_{\alpha}}-1)e = x, \quad (\gamma -1)e = y.
\end{equation}
It turns out that such an assignment gives rise to an isomorphism \[
    H^1(D) \cong \mathrm{Ext}_{(\varphi, \Gamma)}^1(\calR_{K, E}, D).
\]
Furthermore, we write $H^1_{\st}(D)$ (resp., $H^1_f(D)$) the subspace of $H^1(D)$, consisting of those semistable (resp., crystalline) extensions $D_{\alpha}$, \emph{i.e.}, those satisfy $\rank_{K_0\otimes_{\Q_p}E}\calD_{\st}(D_{\alpha}) = \rank_{K_0 \otimes_{\Q_p} E} \calD_{\st}(D) +1$ (resp., $\rank_{K_0\otimes_{\Q_p}E}\calD_{\cris}(D_{\alpha}) = \rank_{K_0 \otimes_{\Q_p} E} \calD_{\cris}(D) +1$). According to \cite[Proposition 1.4.2]{Benois}, if $\rho\in \Rep_{K}(E)$, then \[
    H^1_{\st}(\D_{\rig}^{\dagger}(\rho)) \cong H^1_{\st}(K, \rho) \quad (\text{resp., } H^1_{f}(\D_{\rig}^{\dagger}(\rho)) \cong H_f^1(K, \rho)),
\]
where \begin{align*}
    H_{\st}^1(K, \rho) &= \ker\left(H^1(K, \rho) \rightarrow H^1(K, \rho\otimes_{\Q_p}\B_{\st})\right) \\
    (\text{resp., }H_{f}^1(K, \rho) &= \ker\left(H^1(K, \rho) \rightarrow H^1(K, \rho\otimes_{\Q_p}\B_{\cris})\right))\footnote{ Here, $\B_{\st}$ and $\B_{\cris}$ are, respectively, Fontaine's semistable and crystalline period rings. }
\end{align*}
is the usual local Bloch--Kato Selmer group. 

To conclude our discussion for general $(\varphi, \Gamma)$-modules, we mention that, if $D$ is semistable,\footnote{ In fact, the condition can be loosen to being \emph{potentially semistable}, but we do not need such a generality here. } then $H_{\st}^1(D)$ and $H_f^1(D)$ can be computed by complexes $C_{\st}^{\bullet}$ and $C_{\cris}^{\bullet}$ respectively (\cite[Proposition 1.4.4]{Benois}). Here, \[
    C_{\st}^{\bullet}(D) = \left[ \begin{tikzcd}
        \calD_{\st}(D) \arrow[rrr, "\text{$a\mapsto (a, (\varphi-1)a, N(a))$}"] &&& \frac{\calD_{\st}(D)}{\Fil_{\dR}^0\calD_{\st}(D)} \oplus \calD_{\st}(D) \oplus \calD_{\st}(D) \arrow[d,"\text{$(a,b,c)\mapsto N(b)-(p\varphi-1)c$}"] \\ &&& \calD_{\st}(D)
    \end{tikzcd} \right]
\]
and 
\begin{align*}
    C_{\cris}^{\bullet}(D) & = \left[\calD_{\cris}(D) \xrightarrow{a\mapsto (a, (\varphi-1)a)} \frac{\calD_{\cris}(D)}{\Fil_{\dR}^0\calD_{\cris}(D)} \oplus \calD_{\cris}(D) \right].
\end{align*}

\subsection{\texorpdfstring{$(\varphi, \Gamma)$}{(phi, Gamma)}-modules of rank 1}\label{subsection: rank 1 (phi, Gamma)}
Recall that $(\varphi, \Gamma)$-modules of rank $1$ can be understood via continuous characters. More precisely, given a continuous character $\delta: K^\times \rightarrow E^\times$ and fix a uniformiser $\varpi\in K$, we can write $\delta = \delta' \delta''$ with $\delta'|_{\calO_{K}^{\times}} = \delta|_{\calO_{K}^{\times}}$, $\delta'(\varpi) = 1$ and $\delta''(\varpi) = \delta(\varpi)$, $\delta''|_{\calO_K^{\times}} = 1$. By local class field theory, $\delta'$ defines a unique one-dimensional Galois representation $\chi_{\delta'}$, \emph{i.e.}, \[
    \chi_{\delta'}: \Gal_K \xrightarrow{\text{local Artin map}} \widehat{K^\times} \cong \calO_K^\times \times \widehat{\Z} \xrightarrow{(a, b)\mapsto \delta'(a)} E^\times, \footnote{ Here, $\widehat{K^\times}$ is the profinite completion of $K^\times$. Note that the isomorphism $\widehat{K^\times} \cong \calO_K^\times \times \widehat{\Z}$ depends on the choice of $\varpi$, which is fixed. }
\] which admits its associated $(\varphi, \Gamma)$-module $\D_{\rig}^{\dagger}(\chi_{\delta'})$. On the other hand, we define $\calR_{K, E}(\delta'') = \calR_{K, E}e_{\delta''}$ such that $\varphi(e_{\delta''}) = \delta(\varpi)e_{\delta''}$ and $\gamma(e_{\delta''}) = e_{\delta''}$. Then, the $(\varphi, \Gamma)$-module associated with $\delta$ is defined to be \[
    \calR_{K, E}(\delta) := \calR_{K, E}(\delta'') \otimes_{\calR_{K, E}}\D_{\rig}^{\dagger}(\chi_{\delta'}).
\]

In particular, the cyclotomic character $\Gal_{K} \rightarrow \calO_{E}^\times$ has the associated $(\varphi, \Gamma)$-module $\D_{\rig}^{\dagger}(\chi_{\cyc})$. By \cite[Lemma 2.13]{Nakamura}, we know that \[
    \D_{\rig}^{\dagger}(\chi_{\cyc}) = \calR_{K, E}(\mathrm{Nm}^{K}_{\Q_p}(z)|\mathrm{Nm}^{K}_{\Q_p}(z)|),
\] where $\mathrm{Nm}^K_{\Q_p}$ is the norm function from $K$ to $\Q_p$.

\begin{Lemma}\label{Lemma: H1st = H1 for a special character}
    Let $\delta: K^\times \rightarrow E^\times$ be the character \[
        \delta(z) = \left( \prod_{\sigma: K \hookrightarrow \overline{\Q}_p} \sigma(z)^{m_{\sigma}}\right) \left| \mathrm{Nm}_{\Q_p}^{K} (z)\right|
    \] such that all $m_{\sigma} \geq 1$. \begin{enumerate}
        \item[$\bullet$] If  $\left(\calD_{\st}(\calR_{K, E}(\delta))^{\vee}(\chi_{\cyc}) \right)^{\varphi =1}$ is nonzero, then the inclusion $H_{\st}^1(\calR_{K, E}(\delta)) \hookrightarrow H^1(\calR_{K, E}(\delta))$ is an isomorphism. 
        \item[$\bullet$] If $\left(\calD_{\st}(\calR_{K, E}(\delta))^{\vee}(\chi_{\cyc}) \right)^{\varphi =1} =0 $, then the inclusion $H_f^1(\calR_{K, E}(\delta)) \hookrightarrow H_{\st}^1(\calR_{K, E}(\delta))$ is an isomorphism. 
    \end{enumerate}
\end{Lemma}
\begin{proof}
    By \cite[Corollary 1.4.5]{Benois}, we have the formula \[
        \dim_E H_{\st}^1(\calR_{K, E}(\delta)) - \dim_E H_f^1(\calR_{K, E}(\delta)) = \dim_{E} \left(\calD_{\st}(\calR_{K, E}(\delta))^{\vee}(\chi_{\cyc}) \right)^{\varphi =1}.
    \] Applying \cite[Proposition 2.1 \& Lemma 2.3]{Rosso}, we know that \[
        \dim_E H_{\st}^1(\calR_{K, E}(\delta)) \leq [K: \Q_p]+1 \quad \text{ and }\quad \dim_E H_f^1(\calR_{K, E}(\delta)) = [K: \Q_p].
    \] The lemma then follows easily. 
\end{proof}

Suppose $\delta: K^\times \rightarrow E^\times$ is a continuous character as in Lemma \ref{Lemma: H1st = H1 for a special character}. Suppose $\calR_{K, E}(\delta)$ is semistable and so $\rank_{K_0 \otimes_{\Q_p}E}\calD_{\st}(\calR_{K, E}(\delta)) = 1$. We fix a $K_0 \otimes_{\Q_p}E$-basis $v_{\delta}$ for $\calD_{\st}(\calR_{K, E}(\delta))$ and define \[
    \beta_{\delta}^* = -\cl(0,0,v_{\delta}), \quad \alpha_{\delta}^* = \cl(v_{\delta}, 0, 0) \in H^1(C_{\st}^{\bullet}(\calR_{K, E}(\delta))) = H_{\st}^1(\calR_{K, E}(\delta)). \footnote{ We use such notations due to \cite[Theorem 1.5.7]{Benois}.}
\] 

\begin{Lemma}\label{Lemma: unqiue number defined by a semistable extension}
    Suppose $\eta: K^\times \rightarrow E^\times$ is a continuous character of the form $\eta(z) = \prod_{\sigma: K \hookrightarrow \overline{\Q}_p} \sigma(z)^{n_{\sigma}}$ with all $n_{\sigma} \leq 0$. Suppose \[
        0 \rightarrow \calR_{K, E}(\delta) \rightarrow D \rightarrow \calR_{K, E}(\eta) \rightarrow 0    
    \] is a semistable extension (in the sense of \S \ref{subsection: (phi, Gamma)}). Then, \[
        \image\left( \partial: H^0(\calR_{K, E}(\eta)) \rightarrow H^1(\calR_{K, E}(\delta))\right) \subset H_{\st}^1(\calR_{K, E}(\delta)).
    \] Moreover, there exists a unique $\calL(D)\in E$ such that \[
        \beta_{\delta}^* + \calL(D) \alpha_{\delta}^* \in \image \partial. 
    \]
\end{Lemma}
\begin{proof}
    First of all, it follows from \cite[Proposition 1.2.7]{Benois} that $\calR_{K, E}(\eta)$ is also semistable. Hence, by applying [\emph{op. cit.}, Proposition 1.4.4], we know that \[
        H^0(\calR_{K, E}(\eta)) = H_{\st}^0(\calR_{K, E}). 
    \] Taking the cohomology of the short exact sequence in the lemma, we have a commutative diagram \[
        \begin{tikzcd}
            H^0(\calR_{K, E}(\eta)) \arrow[d, equal] \arrow[r, "\partial"] & H^1(\calR_{K, E}(\delta))\\
            H_{\st}^0(\calR_{K, E}(\eta)) \arrow[r, "\partial"] & H_{\st}^1(\calR_{K, E}(\delta)) \arrow[u, hook]
        \end{tikzcd},
    \] which shows the first claim.
    
    Since $\calR_{K, E}(\eta)$ is semistable and it is of rank $1$ over $\calR_{K,E}$, it is crystalline and $\calD_{\st}(\calR_{K, E}(\eta)) = \calD_{\cris}(\calR_{K, E}(\eta))$. This is because the monodromy operator is nilpotent. We consider the commutative diagram \[
        \begin{tikzcd}[column sep = tiny]
            0 \arrow[r] & \scalemath{0.7}{ \calD_{\st}(\calR_{K, E}(\delta)) } \arrow[r]\arrow[d] & \scalemath{0.7}{ \calD_{\st}(D) } \arrow[r]\arrow[d] & \scalemath{0.8}{ \calD_{\st}(\calR_{K, E}(\eta)) }\arrow[r]\arrow[d] & 0\\
            0 \arrow[r] & \scalemath{0.7}{ \frac{\calD_{\st}(\calR_{K, E}(\delta))}{\Fil_{\dR}^0 \calD_{\st}(\calR_{K, E}(\delta))} \oplus \calD_{\st}(\calR_{K, E}(\delta))^2 }  \arrow[r] \arrow[d] & \scalemath{0.7}{ \frac{\calD_{\st}(D)}{\Fil_{\dR}^0\calD_{\st}(D)} \oplus \calD_{\st}(D)^2 } \arrow[r]\arrow[d] & \scalemath{0.7}{ \frac{\calD_{\st}(\calR_{K, E}(\eta))}{\Fil_{\dR}^0\calD_{\st}(\calR_{K, E}(\eta))} \oplus \calD_{\st}(\calR_{K, E}(\eta))^2 } \arrow[r]\arrow[d] & 0\\
            0 \arrow[r] & \scalemath{0.7}{ \calD_{\st}(\calR_{K, E}(\delta)) } \arrow[r] & \scalemath{0.7}{ \calD_{\st}(D) } \arrow[r] & \scalemath{0.7}{ \calD_{\st}(\calR_{K, E}(\eta)) } \arrow[r] & 0
        \end{tikzcd}
    \] induced by the short exact sequence in the lemma, where the rows are exact and the columns are the semistable complexes. Let $v_{\eta}$ be the element in $\calD_{\st}(\calR_{K, E}(\eta))$ that gives rise to the basis of $H^0(\calR_{K, E}(\eta))$ as in \cite[Proposition 2.1]{Rosso}. In particular, $v_{\eta} \in \Fil_{\dR}^0 \calD_{\st}(\calR_{K, E}(\eta))$ and $\varphi(v_{\eta}) = v_{\eta}$. Using the relation $N \varphi = p \varphi N$, one deduces that $1$ and $p^{-1}$ are Frobenius eigenvalues of $\calD_{\st}(D)$. We choose a lift $\widetilde{v}_{\eta}\in \calD_{\st}(D)$ such that $\varphi(\widetilde{v}_{\eta}) = \widetilde{v}_{\eta}$. This then implies that $N(\widetilde{v}_{\eta})$ has Frobenius eigenvalue $p^{-1}$. The commutativity of the diagram then yields \[
        \begin{tikzcd}
            \widetilde{v}_{\eta} \arrow[r, mapsto]\arrow[d, mapsto] & v_{\eta}\arrow[d, mapsto]\\
            (\widetilde{v}_{\eta}, 0, N(\widetilde{v}_{\eta})) \arrow[r, mapsto] & 0 
        \end{tikzcd}.
    \] Applying the exactness of the middle row, we see that \[
        (\widetilde{v}_{\eta}, 0, N(\widetilde{v}_{\eta})) = a(v_{\delta}, 0, 0) - b(0, 0, v_{\delta}).
    \]
    Since $N(\widetilde{v}_{\eta})$ is a basis for the Frobenius eigensubspace of $\calD_{\st}(D)$ on which $\varphi$ acts via $p^{-1}$, we see that $b$ is invertible. 
    We then conclude that \[
        \partial: H^0(\calR_{K, E}(\eta)) \rightarrow H_{\st}^1(\calR_{K, E}(\delta)), \quad \cl(v_{\eta}) \mapsto a \alpha_{\delta}^* + b\beta_{\delta}^*.
    \]
    Therefore, $\calL(D) := a/b$.
\end{proof}
\section{Greenberg--Benois \texorpdfstring{$\calL$}{L}-invariants}\label{section: GB}

In this section, we first discuss the construction of Greenberg--Benois $\calL$-invariant over $\Q$ in \S \ref{subsection: GB over Q} by summarising Benois's construction in \cite[\S 2]{Benois} (in particular [\emph{op. cit.}, (26)]). Then we follow the strategy in \cite{Rosso}, generalising Benois's construction to general number fields by considering inductions of Galois representations (\S \ref{subsection: GB over unmber fields}).

\subsection{Greenberg--Benois \texorpdfstring{$\calL$}{L}-invariants over \texorpdfstring{$\Q$}{Q}}\label{subsection: GB over Q}
To define Greenberg--Benois $\calL$-invariants over $\Q$, we start with a Galois representation \[
    \rho: \Gal_{\Q} \rightarrow \GL_n(E),
\] which is unramified outside a finite set of places. We denote by \[
    S = \{\ell: \rho|_{\Gal_{\Q_{\ell}}} \text{ is ramified}\} \cup \{p, \infty\}
\] and let $\Q_{S}$ be the maximal extension of $\Q$ that is unramified outside $S$.

Recall the Bloch--Kato Selmer group associated with $\rho$: Given $v\in S$, define the local Selmer groups \[
    H_f^1(\Q_{v}, \rho) = \left\{\begin{array}{ll}
        \ker\left(H^1(\Q_{\ell}, \rho) \rightarrow H^1(I_{\ell}, \rho)\right), & \text{ if }v=\ell \nmid p\infty,  \\
        \ker(H^1(\Q_p, \rho) \rightarrow H^1(\Q_p, \rho\otimes_{\Q_p}\B_{\cris})), & \text{ if }v=p, \\
        H^1(\R, \rho), & \text{ if } v=\infty,
    \end{array}\right.
\] where $I_{\ell}$ stands for the inertia group at $\ell$. Then, the Bloch-Kato Selmer group associated with $\rho$ is defined to be \[
    H_f^1(\Q, \rho) := \ker\left( H^1(\Gal_{\Q_S}, \rho) \rightarrow \bigoplus_{v\in S} \frac{H^1(\Q_v, \rho)}{H_f^1(\Q_v, \rho)}\right).
\]

Let $\rho_p := \rho|_{\Gal_{\Q_p}}$. We follow \cite[\S 2.1.2, 2.1.4]{Benois} and proceed with the following conditions:\begin{enumerate}
    \item[(B1)] The local representation $\rho_p$ is semistable with Hodge--Tate weights $k_1 \leq k_2 \leq \cdots \leq k_n$, giving rise to the \emph{de Rham filtration} $\Fil_{\dR}^{\bullet} \D_{\st}(\rho)$.
    \item[(B2)] The Frobenius action on $\D_{\st}(\rho_p)$ is semisimple at $1$ and $p^{-1}$.
    \item[(GB1)] $H_f^1(\Q, \rho) = 0 = H_f^1(\Q, \rho^{\vee}(1))$.\footnote{Here, by confusing $\rho$ with its underlying vector space, $\rho^{\vee} = \Hom(\rho, E)$ is the dual representation of $\rho$ and $\rho^{\vee}(1)$ is the twist of $\rho^{\vee}$ by the cyclotomic character.}  
    \item[(GB2)] $H^0(\Gal_{\Q_S}, \rho) = 0 = H^0(\Gal_{\Q_S}, \rho^{\vee}(1))$.
    
    \item[(GB3)] The associated $(\varphi, \Gamma)$-module $\D_{\rig}^{\dagger}(\rho_p)$ has no saturated subquotient\footnote{ Here, by `saturated', we mean the following: A \emph{saturated} $(\varphi, \Gamma)$-submodule of a $(\varphi, \Gamma)$-module is a $(\varphi, \Gamma)$-submodule that has a torsion-free quotient. A \emph{saturated} subquotient is a subquotient arising from saturated $(\varphi, \Gamma)$-submodules; in particular, a saturated subquotient is torsion-free. } isomorphic to $U_{k,m}$ with $k\geq 1$ and $m\geq 0$ (\cite[\S 2.1.2]{Benois}), where $U_{k,m}$ is the unique crystalline $(\varphi, \Gamma)$-module sitting in a non-split short exact sequence \[
        0 \rightarrow \calR_{\Q_p, E}(|z|z^k) \rightarrow U_{k,m} \rightarrow \calR_{\Q_p, E}(z^{-m}) \rightarrow 0.
    \]
\end{enumerate}

Given a \emph{regular submodule} $D\subset \D_{\st}(\rho_p)$, \emph{i.e.}, a $(\varphi, N)$-submodule such that $\D_{\st}(\rho_p) = D \oplus \Fil_{\dR}^0\D_{\st}(\rho_p)$, Benois defines a five-step filtration \begin{equation}\label{eq: Benois's 5-step filtration; (phi,N)-module side}
    D^{\GB}_i := \left\{ \begin{array}{cl}
        0, & i=-2, \\
        (1-p^{-1}\varphi^{-1})D + N(D^{\varphi = 1}), & i=-1,\\
        D, & i=0,\\
        D + \D_{\st}(\rho_p)^{\varphi=1}\cap N^{-1}(D^{\varphi=p^{-1}}), & i=1,\\
        \D_{\st}(\rho_p), & i=2.
    \end{array}\right. 
\end{equation}
Such a filtration then yields a filtration on $\D_{\rig}^{\dagger}(\rho_p)$ by \[
    \Fil_i^{\GB} \D_{\rig}^{\dagger}(\rho_p) = \D_{\rig}^{\dagger}(\rho_p) \cap \left( D_i^{\GB} \otimes_{\Q_p} \calR_{\Q_p, E}[1/t]\right),
\] where $t = \log(1+T) \in \calR_{\Q_p, E}$.

Using this filtration, we define the \emph{exceptional subquotient}\[
    W := \Fil_1^{\GB} \D_{\rig}^{\dagger}(\rho_p) / \Fil_{-1}^{\GB} \D_{\rig}^{\dagger}(\rho_p).
\] By \cite[Proposition 2.1.7]{Benois}, we have \[
    \begin{array}{rcl}
        W \cong W_0 \oplus W_1 \oplus M & & \rank W_0 = \dim_E H^0(W^{\vee}(1)),  \\
        \Gr_0^{\GB}\D_{\rig}^{\dagger}(\rho_p) \cong W_0 \oplus M_0 & \text{ with } &  \rank W_1 = \dim_E H^0(W),\\
        \Gr_1^{\GB}\D_{\rig}^{\dagger}(\rho_p) \cong W_1 \oplus M_1 & & \rank M_0 = \rank M_1,
    \end{array}
\]
where $M$, $M_0$, and $M_1$ sit inside a short exact sequence \[
    0 \rightarrow M_0 \rightarrow M \rightarrow M_1 \rightarrow 0.
\]
Moreover, one has  \begin{align*}
    H^1(W) & = \coker \left(H^1(\Fil_{-1}^{\GB}\D_{\rig}^{\dagger}(\rho_p)) \rightarrow H^1(\Fil_1^{\GB}\D_{\rig}^{\dagger}(\rho_p)) \right),\\
    H_f^1(W) & = \coker \left(H_f^1(\Fil_{-1}^{\GB}\D_{\rig}^{\dagger}(\rho_p)) \rightarrow H_f^1(\Fil_1^{\GB}\D_{\rig}^{\dagger}(\rho_p)) \right),
\end{align*} and $\dim_E H^1(W)/H_f^1(W) = e_D = \rank M_0 + \rank W_0 +\rank W_1$ (\cite[\S 2.2.1]{Benois}).

Under the assumption (GB1) and (GB2), one applies Poitou--Tate exact sequence and deduces an isomorphism \[
    H^1(\Gal_{\Q_S}, \rho) \cong \bigoplus_{v\in S} \frac{H^1(\Q_v, \rho)}{H_f^1(\Q_v, \rho)}.
\]
Note that the latter space contains an $e_D$-dimensional subspace $\frac{H^1(W)}{H_f^1(W)} \cong \frac{H^1(\Fil_1^{\GB}\D_{\rig}^{\dagger}(\rho_p))}{H^1_f(\Q_p, \rho)}$. We then define $H^1(D, \rho)$ to be the image of $\frac{H^1(W)}{H_f^1(W)}$ in $H^1(\Gal_{\Q_S}, \rho)$.

To define the $\calL$-invariant, we further assume that \begin{enumerate}
    \item[(GB4)] $W_0 = 0$ and the Hodge--Tate weights for $\Gr_1^{\GB} \D_{\rig}^{\dagger}(\rho_p)$ are positive (see \cite[Proposition 1.5.9]{Benois}). 
\end{enumerate}
Benois shows that there is a decomposition (\cite[\S 2.1.9]{Benois}, see also the discussion in \cite[\S 1.2]{HJ-LInvariant}) \[
    H^1(\Gr_1^{\GB}\D_{\rig}^{\dagger}(\rho_p)) \cong H_f^1(\Gr_1^{\GB}\D_{\rig}^{\dagger}(\rho_p)) \oplus H_c^1(\Gr_1^{\GB}\D_{\rig}^{\dagger}(\rho_p))
\] and isomorphisms \[
    H_f^1(\Gr_1^{\GB}\D_{\rig}^{\dagger}(\rho_p)) \cong \calD_{\cris}(\Gr_1^{\GB}\D_{\rig}^{\dagger}(\rho_p)) \cong H_c^1(\Gr_1^{\GB}\D_{\rig}^{\dagger}(\rho_p)).
\]
There are natural morphisms $\varrho_{D, ?}: H^1(D, \rho) \rightarrow \calD_{\cris}(\Gr_1^{\GB}\D_{\rig}^{\dagger}(\rho_p))$ (for $? \in \{f, c\}$) making the diagram \[
    \begin{tikzcd}
        \calD_{\cris}(\Gr_1^{\GB}\D_{\rig}^{\dagger}(\rho_p)) \arrow[r, "\cong"] & H_f^1(\Gr_1^{\GB}\D_{\rig}^{\dagger}(\rho_p))\\
        H^1(D, \rho)\arrow[u, "\varrho_{D, f}"]\arrow[d, "\varrho_{D, c}"'] \arrow[r] & H^1(\Gr_1^{\GB} \D_{\rig}^{\dagger}(\rho_p))\arrow[u]\arrow[d]\\
        \calD_{\cris}(\Gr_1^{\GB}\D_{\rig}^{\dagger}(\rho_p)) \arrow[r, "\cong"] & H_c^1(\Gr_1^{\GB}\D_{\rig}^{\dagger}(\rho_p))
    \end{tikzcd}
\]
commutative. Under the assumption of (GB4), Benois shows that $\varrho_{D, c}$ is an isomorphism and so one can define the \textbf{\textit{Greenberg--Benois $\calL$-invariant}} attached to $\rho$ (with respect to $D$) as \[
    \calL_{\GB}(\rho) = \calL_{\GB}(\rho, D) := \det\left( \varrho_{D, f} \circ \varrho_{D, c}^{-1}\right)\in E
\]

\subsection{Greenberg--Benois \texorpdfstring{$\calL$}{L}-invariants over general number fields}\label{subsection: GB over unmber fields}
To define the Greenberg--Benois $\calL$-invariants over general number fields, we follows the idea in \cite{Rosso} (see also \cite{Hida-HilberLinvariant}) and consider the induction of a Galois representation. More precisely, let $F$ be a number field and suppose we are given a Galois representation \[
    \rho: \Gal_F \rightarrow \GL_n(E),
\]
where $E$ is (again) a finite extension of $\Q_p$. We shall consider the induction $\Ind^{\Q}_F \rho$ and define $S$ similarly as before.

Assume the following conditions hold for $\rho$: \begin{enumerate}
    \item[(B1)] For each place $\frakp |p$ in $F$, $\rho_{\frakp} = \rho|_{\Gal_{F_{\frakp}}}$ is semistable with Hodge--Tate weights $k_{\frakp, \sigma, 1} \leq k_{\frakp, \sigma, 2}\leq \cdots \leq k_{\frakp, \sigma, n}$ where $\sigma: F_{\frakp} \hookrightarrow \overline{\Q}_p$. 
    \item[(B2)] For each place $\frakp |p$ in $F$, the Frobenius action on $\D_{\st}(\rho_{\frakp})$ is semistable at $1$ and $p^{-1}$. 
    \item[(GB1)] $H_f^1(\Q, \Ind^{\Q}_{F} \rho) = 0 = H_f^1(\Q, \Ind_F^{\Q} \rho^{\vee}(1))$.
    \item[(GB2)] $H^0(\Gal_{\Q_S}, \rho) = 0 = H^0(\Gal_{\Q_S}, \rho^{\vee}(1))$.
    \item[(GB3)] The associated $(\varphi, \Gamma)$-module $\D_{\rig}^{\dagger}((\Ind_F^{\Q}\rho)_p) = \bigoplus_{\frakp|p} \D_{\rig}^{\dagger}(\rho_{\frakp})$ has no saturated subquotient isomorphic to $U_{k,m}$ with $k\geq 1$ and $m\geq 0$ (\cite[\S 2.1.2]{Benois}).
\end{enumerate}

For every $\frakp|p$, choose a regular subomdule $D_{\frakp} \subset \D_{\st}(\rho_{\frakp})$. Then, $D := \bigoplus_{\frakp |p} D_{\frakp } \subset \bigoplus_{\frakp | p}\D_{\st}(\rho_{\frakp}) = \D_{\rig}^{\dagger}((\Ind^{\Q}_{F}\rho)_p)$ is a regular submodule.\footnote{ Note that $\D_{\st}(\Ind^{\Q_p}_{F_{\frakp}}\rho_{\frakp})$ is nothing but $\D_{\st}(\rho_{\frakp})$ (a priori a $K_0 \otimes_{\Q_p}E$-module) viewing as a $E$-vector space. } Moreover, if $W_0$, $M_0$, $M_1$ (resp., $W_{\frakp, 0}$, $M_{\frakp, 0}$, $M_{\frakp, 1}$) are the corresponding subquotients of $\D_{\rig}^{\dagger}((\Ind^{\Q}_{F}\rho)_p)$ (resp., $\D_{\rig}^{\dagger}(\rho_{\frakp})$) with respect to $D$ (resp., $D_{\frakp}$), then we have decompositions \[
    W_0 = \bigoplus_{\frakp | p}W_{\frakp, 0}, \quad M_0 = \bigoplus_{\frakp| p} M_{\frakp, 0}, \quad M_1 = \bigoplus_{\frakp|p} M_{\frakp, 1}.
\]
Hence, by assuming \begin{enumerate}
    \item[(GB4)] $W_{\frakp, 0} =0$ for every $\frakp |p$ and the Hodge--Tate weights for $\Gr_1^{\GB} \D_{\rig}^{\dagger}((\Ind_F^{\Q} \rho)_p)$ are all positive, 
\end{enumerate}
we may then follow the same recipe and define the \textbf{\textit{Greenberg--Benois $\calL$-invariant}} attached to $\rho$ (with respect to $\{D_{\frakp}\}_{\frakp|p}$) \[
    \calL_{\GB}(\rho) = \calL_{\GB}(\rho, \{D_{\frakp}\}_{\frakp|p}) := \calL_{\GB}(\Ind_{F}^{\Q}\rho, D)\in E.
\]
\section{Fontaine--Mazur \texorpdfstring{$\calL$}{L}-invariants}\label{section: FM}

To define the Fontaine--Mazur $\calL$-invariants, we fix a finite extension $K$ over $\Q_p$. We shall be considering Galois representations \[
    \rho: \Gal_K \rightarrow \GL_n(E),
\] 
where $E$ is (again) a finite extension of $\Q_p$. In what follows, we consider the $(\varphi, N)$-module $\D_{\st}(\rho)$ associated with $\rho$. Note that, if $K_0$ is the maximal unramified extension of $\Q_p$ in $K$, then $\D_{\st}(\rho)$ is a priori a $K_0$-vector space. However, we shall linearise everything by base change to $E$. 

Let $q$ be the order of the residue field of $K$. We further assume $\rho$ enjoys the following properties: 
\begin{enumerate}
    \item[(B1)] The representation $\rho$ is semistable with Hodge--Tate weights $k_{\sigma, 1} \leq k_{\sigma, 2} \leq \cdots \leq k_{\sigma, n-1}\leq k_{\sigma, n}$, where $\sigma : K \hookrightarrow \overline{\Q}_p$. The Hodge--Tate weights give rise to the de Rham filtration $\Fil_{\dR}^{\bullet} \D_{\st}(\rho) = [\Fil^{k_{\bullet, 1}}_{\dR} \D_{\st}(\rho) \supset \Fil^{k_{\bullet, 2}}_{\dR} \D_{\st}(\rho) \supset \cdots \supset \Fil_{\dR}^{k_{\bullet, n}}\D_{\st}(\rho)]$.
    \item[(B2)] The linearised Frobenius eigenvalues on $\D_{\st}(\rho)$ are $q^m, ..., q^{m-n+1}$.
    \item[(FM1)] Let $D_{(\varphi, N)}^{(i)}$ be the eigenspace in $\D_{\st}(\rho)$ on which the Frobenius acts via $q^{m-i}$ and we assume that the induced monodromy operator $N$ on $D_{(\varphi, N)}^{(i)}$ gives an isomorphism \[
        N : D_{(\varphi, N)}^{(i)} \rightarrow D_{(\varphi, N)}^{(i+1)}.
    \]
    \item[(FM2)] Define \emph{Frobenius filtration} $\Fil_{\bullet}^{\varphi}\D_{\st}(\rho)$ by $\Fil_j^{\varphi}\D_{\st}(\rho) := \sum_{i>n-1-j} D_{(\varphi, N)}^{(i)}$ and assume the orthogonality \[
        \D_{\st}(\rho) = \Fil_{\dR}^{k_{\bullet, i}} \D_{\st}(\rho) \oplus \Fil_i^{\varphi}\D_{\st}(\rho).
    \]
\end{enumerate}

\begin{Lemma}\label{Lemma: dR graded piece mapping into phi graded piece}
    Keep the notations and the assumptions as above. We abuse the notation and denote by $\Gr_{\dR}^{n-1}\D_{\st}(\rho) := \Fil^{k_{\bullet, n-1}}_{\dR}\D_{\st}(\rho)/\Fil_{\dR}^{k_{\bullet, n}} \D_{\st}(\rho)$. Then, we have an inclusion \[
        \Gr_{\dR}^{n-1} \D_{\st}(\rho) \hookrightarrow D_{(\varphi, N)}^{(0)} \oplus D_{(\varphi, N)}^{(1)}
    \]
\end{Lemma}
\begin{proof}
    Indeed, we have a sequence of identifications \begin{align*}
        D_{(\varphi, N)}^{(0)} \oplus D_{(\varphi, N)}^{(1)} & = \frac{\Fil_n^{\varphi} \D_{\st}(\rho)}{\Fil_{n-2}^{\varphi}\D_{\st}(\rho)} \\
        & = \frac{\Fil_n^{\varphi} \D_{\st}(\rho) \oplus \Fil_{\dR}^{k_{\bullet, n-2}}\D_{\st}(\rho)}{\Fil_{n-2}^{\varphi}\D_{\st}(\rho) \oplus \Fil_{\dR}^{k_{\bullet, n-2}}\D_{\st}(\rho)}\\
        & = \frac{\Fil_n^{\varphi} \D_{\st}(\rho) \oplus \Fil_{\dR}^{k_{\bullet, n-2}}\D_{\st}(\rho)}{\D_{\st}(\rho)}\\
        & = \frac{\Fil_n^{\varphi} \D_{\st}(\rho) \oplus \Fil_{\dR}^{k_{\bullet, n-2}}\D_{\st}(\rho)}{\Fil_n^{\varphi} \D_{\st}(\rho) \oplus \Fil_{\dR}^{k_{\bullet, n}}\D_{\st}(\rho)}\\
        & = \frac{\Fil_{\dR}^{k_{\bullet, n-2}}\D_{\st}(\rho)}{ \Fil_{\dR}^{k_{\bullet, n}} \D_{\st}(\rho)},
    \end{align*} where the third and the forth identifications follows from the orthogonality assumption.
\end{proof}

\begin{Lemma}\label{Lemma: rank of D_(phi, N)^{(i)}}
    For every $i$, we have \[
        \rank_{K_0 \otimes_{\Q_p}E} D_{(\varphi, N)}^{(i)} = 1.
    \] Moreover, $m <k_{\sigma, n}$ for every $\sigma: K \hookrightarrow \overline{\Q}_p$.
\end{Lemma}
\begin{proof}
    Consider the twisted Galois representation $\rho(m)$. One can similarly define the Frobenius filtration $\Fil_{\bullet}^{\varphi}\D_{\st}(\rho(m))$ and we denote by $D_{(\varphi, N)}^{(i)}(m)$ the graded pieces. Since each $\Fil^{\varphi}_i \D_{\st}(\rho(m))$ is a $(\varphi, N)$-module, \cite[Proposition 1.2.7 (ii)]{Benois} implies that we have an associated filtration $\Fil_{\bullet}\D_{\rig}^{\dagger}(\rho(m))$ such that $\calD_{\st}(\Fil_{\bullet}\D_{\rig}^{\dagger}(\rho(m))) = \Fil^{\varphi}_{i}\D_{\st}(\rho(m))$.

    Consider $\Gr_{n}\D_{\rig}^{\dagger}(\rho(m))$. One sees that \[
        \calD_{\st}(\Gr_{n}\D_{\rig}^{\dagger}(\rho(m))) = D_{(\varphi, N)}^{(0)}(m),
    \] on which the semistable Frobenius acts via $1$. Hence, by \cite[Proposition 2.4]{Rosso}, $\Gr_{n}\D_{\rig}^{\dagger}(\rho(m))$ is crystalline and \[
        \Gr_{n}\D_{\rig}^{\dagger}(\rho(m)) \cong \calR_{K, E}(\delta) \quad  \text{ with } \quad \delta(z) = \prod_{\sigma: K \hookrightarrow \overline{\Q}_p} \sigma(z)^{-k_{\sigma, n}+m}.
    \] This shows that $\rank_{K_0 \otimes_{\Q_p} E} D_{(\varphi, N)}^{(0)}(m) = 1$. Using the formula in \emph{loc. cit.}, one also sees that $k_{\sigma, n} > m$.

    Since $\rank_{K_0 \otimes_{\Q_p} E} D_{(\varphi, N)}^{(0)}(m) = 1$, we see that $\rank_{K_0 \otimes_{\Q_p} E} D_{(\varphi, N)}^{(0)} = 1$. The result then can be concluded by applying (FM1). 
\end{proof}

Thanks to Lemma \ref{Lemma: dR graded piece mapping into phi graded piece} and Lemma \ref{Lemma: rank of D_(phi, N)^{(i)}}, we can now define the \emph{Fontaine--Mazur $\calL$-invariant}. Let $v_0$ be a $K_0 \otimes_{\Q_p} E$-basis for $D_{(\varphi, N)}^{(0)}$ and let $v_1 := Nv_0$, which is a a $K_0 \otimes_{\Q_p} E$-basis for $D_{(\varphi, N)}^{(1)}$. The \textbf{\textit{Fontaine--Mazur $\calL$-invariant}} attached to $\rho$ is then defined to be $\calL_{\FM}(\rho) \in K_0 \otimes_{\Q_p}E$ such that \[
    v_0 - \calL_{\FM}(\rho)v_1 \in \Gr_{\dR}^{n-1}\D_{\st}(\rho).
\]

\begin{Remark}\label{Remark: Besser--de Shalit's version of FM L-invariants}
    In fact, if we write $\Gr_{\dR}^i \D_{\st}(\rho) := \Fil_{\dR}^{k_i}\D_{\st}(\rho)/\Fil_{\dR}^{k_{i+1}} \D_{\st}(\rho)$, then a similar argument as in Lemma \ref{Lemma: dR graded piece mapping into phi graded piece} shows that \[
        \Gr_{\dR}^{n-i} \D_{\st}(\rho) \hookrightarrow D_{(\varphi, N)}^{(i-1)} \oplus D_{(\varphi, N)}^{(i)}.
    \]
    By using this inclusion, one can similarly define the \emph{$i$-th Fontaine--Mazur $\calL$-operator} attached to $\rho$ to be $\calL_{\FM}^{(i)}(\rho)\in K_0 \otimes_{\Q_p}E$ such that $v_{i-1} - \calL_{\FM}^{(i)}v_{i}\in \Gr_{\dR}^{n-i}\D_{\st}(\rho)$, where $v_j = N^j v$. Such a strategy was taken in \cite{Besser--de_Shalit}. However, it is believed that $\calL_{\FM}^{(0)}(\rho) = \calL_{\FM}(\rho)$ should determine all the other $\calL_{\FM}^{(i)}(\rho)$'s (see, for example, [\emph{op. cit.}, \S 4.3.2]). Hence, we focus on $\calL_{\FM}(\rho)$. Moreover, one shall see, in what follows, that it is $\calL_{\FM}(\rho)$ we can relate to Greenberg--Benois $\calL$-invariants. 
\end{Remark}

\section{Comparing the two \texorpdfstring{$\calL$}{L}-invariants}\label{section: comparison}

The aim of this section is to prove the comparison theorem (Theorem \ref{Theorem: main; comparison}). However, as aforementioned, to define $\calL$-invariants, there are some constraints one needs to put on the Galois representations. For reader's convenience, we collect all the assumptions in \S \ref{subsection: assumptions} and briefly discuss a folklore about these assumptions.

\subsection{Assumptions on the Galois representation}\label{subsection: assumptions}

Let $F$ be a number field and let $E$ be a finite extension of $\Q_p$ such that, for every prime ideal $\frakp$ in $\calO_F$ sitting above $p$, $F_{\frakp} \subset E$. Write $F_{\frakp, 0}$ for the maximal unramified extension of $\Q_p$ in $F_{\frakp}$; we further assume that $F_{\frakp, 0} = \Q_p$ for every $\frakp$.  Suppose we are given a Galois representation \[
    \rho: \Gal_F \rightarrow \GL_n(E)
\]
that is unramified outside a finite set of places. Let $S$ be the set of places in $F$ such that $\rho$ ramifies. We make the following assumptions: \begin{enumerate}
    \item[(I)] Basic assumptions:\begin{enumerate}
        \item[(B1)] For any prime ideal $\frakp \subset \calO_{F}$ sitting above $p$, $\rho_{\frakp} := \rho|_{\Gal_{F_{\frakp}}}$ is semistable with Hodge--Tate weights $0\leq k_{\frakp, \sigma, 1} \leq k_{\frakp, \sigma, 2} \leq \cdots \leq k_{\frakp, \sigma, n-1} \leq k_{\frakp, n}$, where $\sigma : F_{\frakp} \hookrightarrow \overline{\Q}_p$.
        \item[(B2)] For any prime ideal $\frakp \subset \calO_{F}$ sitting above $p$, the Frobenius eigenvalues on $\D_{\st}(\rho_{\frakp})$ are $p^m$, ..., $p^{m-n+1}$ such that $k_{\frakp, \sigma, n} > m > k_{\frakp, \sigma, n-1}$ where the first inequality is always guaranteed by Lemma \ref{Lemma: rank of D_(phi, N)^{(i)}}.\footnote{ Here, $m$ does not depend on $\frakp$ and $\sigma$.}
    \end{enumerate}

    \item[(II)] Fontaine--Mazur assumptions: \begin{enumerate}
        \item[(FM1)] For any $\frakp|p$, let $D_{\frakp, (\varphi, N)}^{(i)}$ be the eigenspace in $\D_{\st}(\rho_{\frakp})$ on which the Frobenius acts via $p^{m-i}$. We assume that the induced monodromy operator $N$ on $D_{\frakp, (\varphi, N)}^{(i)}$ gives an isomorphism \[
            N: D_{\frakp, (\varphi, N)}^{(i)} \rightarrow D_{\frakp, (\varphi, N)}^{(i+1)}.
        \]
        \item[(FM2)] Define $\Fil_j^{\varphi} \D_{\st}(\rho_{\frakp}) := \sum_{i>n-1-j} D_{\frakp, (\varphi, N)}^{(i)}$ and we call the ascending filtration $\Fil_{\bullet}^{\varphi} \D_{\st}(\rho_{\frakp})$ the \emph{Frobenius filtration} on $\D_{\st}(\rho_{\frakp})$. We assume the orthogonality \[
            \D_{\st}(\rho_{\frakp}) = \Fil_{\dR}^{k_{\frakp, \bullet, i}}\D_{\st}(\rho_{\frakp}) \oplus \Fil_i^{\varphi} \D_{\st}(\rho_{\frakp}).
        \]
    \end{enumerate}

    \item[(III)] Greenberg--Benois assumptions: \begin{enumerate}
        \item[(GB1)] Vanishing of the Bloch--Kato Selmer groups \[
            H_f^1(\Q, \Ind^{\Q}_{F} \rho(m)) = H^1_f(\Q, \Ind^{\Q}_{F}\rho^{\vee}(1-m)) = 0.
        \] 
        \item[(GB2)]  Vanishing of the zero-degree Galois cohomology \[
            H^0(\Gal_{\Q_S}, \Ind^{\Q}_{F} \rho(m)) = H^0(\Gal_{\Q_S},\Ind^{\Q}_F\rho^{\vee}(1-m)) =0.
        \]
        \item[(GB3)] The associated $(\varphi, \Gamma)$-module $\D_{\rig}^{\dagger}((\Ind_F^{\Q}\rho(m))_p) = \bigoplus_{\frakp|p} \D_{\rig}^{\dagger}(\rho_{\frakp}(m))$ doesn't admit a subquotient of the form $U_{k,r}$ with $k\geq 1$ and $r\geq 0$ (\cite[\S 2.1.2]{Benois}). 
        \item[(GB4)] For any $\frakp|p$, the space $W_{\frakp, 0}$ for $\rho_{\frakp}(m)$ vanishes (see \cite[Proposition 2.1.7]{Benois} or \cite[pp. 1238]{Rosso}).
    \end{enumerate}
\end{enumerate}

\begin{Remark}\label{Remark: triangulation}
    For every $\frakp | p$, the Frobenius filtration $\Fil_{\bullet}^{\varphi}\D_{\st}(\rho_{\frakp})$ defines a filtration $\Fil_{\bullet} \D_{\rig}^{\dagger}(\rho)$ (similar as in the proof of Lemma \ref{Lemma: rank of D_(phi, N)^{(i)}}). One observes that the graded pieces $\Gr_{i}\D_{\rig}^{\dagger}(\rho)$ of this filtration are all of rank $1$ over $\calR_{F_{\frakp}, E}$ (by \cite[Proposition 1.2.7 (ii)]{Benois}). In particular, $\Fil_{\bullet}\D_{\rig}^{\dagger}(\rho)$ is a triangulation of $\D_{\rig}^{\dagger}(\rho)$. In fact, we have the following description for the graded pieces \[
        \Gr_i \D_{\rig}^{\dagger}(\rho) = \calR_{F_{\frakp}, E}(\delta_{n-i}) 
    \] where \[
        \delta_{n-i}(z) = \left(\prod_{\sigma: F_{\frakp} \hookrightarrow \overline{\Q}_p} \sigma(z)^{-k_{\frakp, \sigma, i}} \right)|\mathrm{Nm}^{F_{\frakp}}_{\Q_p}z|^{-(m-n+i)}.
    \]
\end{Remark}

\begin{Remark}\label{Remark: regular submodule in our case}
    For every $\frakp | p$, since $k_{\frakp, \sigma,  n} > m > k_{\frakp,\sigma, n-1} $, we see that $\Fil_{\dR}^{k_{\frakp, \bullet, n-1}-m}\D_{\st}(\rho_{\frakp}(m)) = \Fil_{\dR}^0 \D_{\st}(\rho_{\frakp}(m))$. Moreover, the orthogonality condition (FM2) implies that $\Fil_{n-1}^{\varphi}\D_{\st}(\rho_{\frakp}(m))$ is a regular $(\varphi, N)$-submodule of $\D_{\st}(\rho_{\frakp}(m))$. Hence, in what follows, we naturally work with $D_{\frakp} := \Fil_{n-1}^{\varphi}\D_{\st}(\rho_{\frakp}(m)) \subset \D_{\st}(\rho_{\frakp}(m))$ and $D = \bigoplus_{\frakp | p}D_{\frakp} \subset \bigoplus_{\frakp |p}\D_{\st}(\rho_{\frakp}(m))$. Moreover, in our situation, we shall see in the proof (\emph{e.g.}, \eqref{eq: dual short exact sequence}) that the corresponding $\Gr_1^{\GB} \D_{\rig}^{\dagger}((\Ind_F^{\Q}\rho)_p)$ has positive Hodge--Tate weights and so we remove such assumption in (GB4).
\end{Remark}

\begin{Remark}\label{Remark: why the assumptions make sense}
    We have many assumptions on our Galois representation $\rho$. On the one hand, one sees that they are necessary in order to attach both $\calL_{\GB}$ and $\calL_{\FM}$ to it. On the other hand, we remark that it is a folklore that they shall appear as Galois representations for automorphic forms of unitary groups whose corresponding Shimura varieties can be $p$-adically uniformised by Drinfeld's upper-half spaces. For example, we are requiring maximal monodromy on our Galois representations. Such a phenomenon is expected to appear for the Galois representations attached to unitary automorphic representations that are Steinberg at $p$. 
\end{Remark}

\subsection{The main theorem}\label{subsection: main theorem}

\begin{Theorem}\label{Theorem: main; comparison}
    Keep the notations and assumptions as above. We have an equality \[
        \calL_{\GB}(\rho(m)) = \prod_{\frakp | p} - \calL_{\FM}(\rho_{\frakp}).
    \]
\end{Theorem}


\begin{proof}
The proof of the theorem is similar to the proof of \cite[Proposition 2.3.7]{Benois}, which relies on the following three steps: 

\vspace{2mm}

\noindent \textbf{Step 1. } Fontaine--Mazur $\calL$-invariants and cohomology of $(\varphi, \Gamma)$-modules.

Consider the triangulation $\Fil_{\bullet}\D_{\rig}^{\dagger}(\rho)$ in Remark \ref{Remark: triangulation}. We define \[
    \widetilde{W}_{\frakp} := \Fil_{n}\D_{\rig}^{\dagger}(\rho)/\Fil_{n-2}\D_{\rig}^{\dagger}(\rho).
\] Hence, $\widetilde{W}_{\frakp}$ sits inside the short exact sequence \[
    \begin{tikzcd}[row sep = tiny]
        0 \arrow[r] & \Gr_{n-1}\D_{\rig}^{\dagger}(\rho) \arrow[d, "\cong"]\arrow[r] & \widetilde{W}_{\frakp} \arrow[r] & \Gr_{n}\D_{\rig}^{\dagger}(\rho) \arrow[d, "\cong"] \arrow[r] & 0\\ & \calR_{F_{\frakp}, E}(\delta_{\frakp, 1}) & & \calR_{F_{\frakp}, E}(\delta_{\frakp, 0})
    \end{tikzcd}
\] for $\delta_{\frakp, i} : F_{\frakp}^\times \rightarrow E^\times$ described as in Remark \ref{Remark: triangulation}.

As a result, $\widetilde{W}_{\frakp}$ defines a class \begin{align*}
    \cl(\widetilde{W}_{\frakp}) & \in  \mathrm{Ext}_{(\varphi, \Gamma)}(\calR_{F_{\frakp}, E}(\delta_{\frakp, 0}), \calR_{F_{\frakp}, E}(\delta_{\frakp, 1})) \\
    & \cong \mathrm{Ext}_{(\varphi, \Gamma)}(\calR, \calR(\delta_{\frakp, 1}\delta_{\frakp, 0}^{-1})) \\
    & \cong H^1(\calR_{F_{\frakp}, E}(\delta_{\frakp, 1}\delta_{\frakp, 0}^{-1})).
\end{align*}
However, by construction, we know that $\widetilde{W}_{\frakp}$ is semistable (since $\calD_{\st}(\widetilde{W}_{\frakp}) = D_{\frakp, (\varphi, N)}^{(0)} \oplus D_{\frakp, (\varphi, N)}^{(1)}$) and so $\cl(\widetilde{W}_{\frakp}) \in H_{\st}^1(\calR_{F_{\frakp}, E}(\delta_{\frakp, 1}\delta_{\frakp, 0}^{-1}))$. Recall that $H_{\st}^1(\calR_{F_{\frakp}, E}(\delta_{\frakp, 1}\delta_{\frakp, 0}^{-1}))$ can be computed via the complex $C_{\st}^{\bullet}(\calR_{F_{\frakp}, E}(\delta_{\frakp, 1}\delta_{\frakp, 0}^{-1}))$ with \[
    \begin{tikzcd}[column sep = tiny]
        \scalemath{0.8}{ \calD_{\st}(\calR_{F_{\frakp}, E}(\delta_{\frakp, 1}\delta_{\frakp, 0}^{-1})) } \arrow[r] & \scalemath{0.8}{ \frac{\calD_{\st}(\calR_{F_{\frakp}, E}(\delta_{\frakp, 1}\delta_{\frakp, 0}^{-1}))}{\Fil_{\dR}^0 \calD_{\st}(\calR_{F_{\frakp}, E}(\delta_{\frakp, 1}\delta_{\frakp, 0}^{-1}))} \oplus \calD_{\st}(\calR_{F_{\frakp}, E}(\delta_{\frakp, 1}\delta_{\frakp, 0}^{-1})) \oplus \calD_{\st}(\calR_{F_{\frakp}, E}(\delta_{\frakp, 1}\delta_{\frakp, 0}^{-1})) } \arrow[d]\\
        & \scalemath{0.8}{ \calD_{\st}(\calR_{F_{\frakp}, E}(\delta_{\frakp, 1}\delta_{\frakp, 0}^{-1}))}
    \end{tikzcd},
\] where the first map is given by $a \mapsto \left(a \mod \Fil_{\dR}^0 \calD_{\st}(\calR_{F_{\frakp}, E}(\delta_{\frakp, 1}\delta_{\frakp, 0}^{-1})), (\varphi-1)a, Na \right)$ while the second arrow is defined by $(a, b, c)\mapsto Nb-(p\varphi-1)c$.

Now, choose a basis $v_{\frakp, 0}\in D_{\frakp, (\varphi, N)}^{(0)}$ over $F_{\frakp, 0}\otimes_{\Q_p} E$ and let $v_{\frakp, 1} := Nv_{\frakp, 0}$, which is a $F_{\frakp, 0}\otimes_{\Q_p} E$-basis for $D_{\frakp, (\varphi, N)}^{(1)}$. We again denote by $v_{\frakp, i}$ for the image of $v_{\frakp, i}$ in $\calD_{\st}(\widetilde{W}_{\frakp}(\delta_{\frakp, 0}^{-1}))$. By the proof of \cite[Proposition 1.4.4 (ii)]{Benois}, we know that the class $\cl(\widetilde{W}_{\frakp})$ in $H^1(C_{\st}^{\bullet}(\calR_{F_{\frakp}, E}(\delta_{\frakp, 1}\delta_{\frakp, 0}^{-1})))$ is given by \[
    \cl(a, b, c) = \cl(a, (\varphi-1)v_{\frakp, 0}, Nv_{\frakp, 0}) = \cl(a, 0 , v_{\frakp, 1}),
\] where $a\in \calD_{\st}(\widetilde{W}_{\frakp}(\delta_{\frakp, 0}^{-1}))$ such that $v_{\frakp, 0} + a \in \Fil_{\dR}^0\calD_{\st}(\widetilde{W}_{\frakp}(\delta_{\frakp, 0}^{-1}))$. After untwisting, $a$ defines an element, still denoted by $a\in \calD_{\st}(\widetilde{W}_{\frakp})$ such that $v_{\frakp, 0} +a \in \Fil_{\dR}^{k_{\frakp, \bullet, n}}\calD_{\st}(\widetilde{W}_{\frakp})$. However, by construction, \[
    \Fil_{\dR}^{k_{\frakp, \bullet, n}}\calD_{\st}(\widetilde{W}_{\frakp}) = \Gr_{\dR}^{n-1}\D_{\st}(\rho) \text{ (notation as in Lemma \ref{Lemma: dR graded piece mapping into phi graded piece})}.
\]
Hence, we conclude that \begin{equation}\label{eq: main computation in Step 1}
    \cl(\widetilde{W}_{\frakp}) = \cl(-\calL_{\FM}(\rho_{\frakp})v_{\frakp, 1}, 0, v_{\frakp, 1})\in H^1(C_{\st}^{\bullet}(\calR_{F_{\frakp}, E}(\delta_{\frakp, 1}\delta_{\frakp, 0}^{-1}))).
\end{equation}

\vspace{2mm}

\noindent \textbf{Step 2. } Computing $\calL_{\GB}(\rho)$.

Next, we would also like to compute the Greenberg--Benois $\calL$-invariant $\calL_{\GB}(\rho)$ via cohomology of $(\varphi, \Gamma)$-modules. As before, because of the decomposition $\D_{\rig}^{\dagger}((\Ind^{\Q}_{F}\rho)_p) = \bigoplus_{\frakp|p}\D_{\rig}^{\dagger}(\Ind_{F_{\frakp}}^{\Q_p}\rho_{\frakp})$, we can study each $\frakp$ individually. Hence, fix $\frakp | p$. Computing the five-step filtration \eqref{eq: Benois's 5-step filtration; (phi,N)-module side} explicitly, we have \[
    D_{\frakp, i}^{\GB} = \left\{ \begin{array}{ll}
        0, & i=-2, \\
        \Fil_{n-2}^{\varphi}\D_{\st}(\rho_{\frakp}(m)), & i=-1,\\
        D_{\frakp}, & i=0\\
        \D_{\st}(\rho_{\frakp}(m)), & i=1,\\
        \D_{\st}(\rho_{\frakp}(m)), & i=2,
    \end{array}\right.
\] which gives rise to a five-step filtration $\Fil_{\bullet}^{\GB}\D_{\rig}^{\dagger}(\rho_{\frakp}(m))$.

Let us simplify the notation and write \[
    W_{\frakp} = \Fil_1^{\GB} \D_{\rig}^{\dagger}(\rho_{\frakp}(m))/\Fil_{-1}^{\GB}\D_{\rig}^{\dagger}(\rho_{\frakp}(m)).
\] Similar as before, we see that $W_{\frakp}$ sits inside the short exact sequence \begin{equation}\label{eq: dual short exact sequence}
    \begin{tikzcd}[row sep = tiny]
        0 \arrow[r] & \Gr_0^{\GB}\D_{\rig}^{\dagger}(\rho_{\frakp}(m)) \arrow[r]\arrow[d, "\cong"] & W_{\frakp} \arrow[r] & \Gr_1^{\GB}\D_{\rig}^{\dagger}(\rho_{\frakp}(m))\arrow[r]\arrow[d, "\cong"] & 0\\
        & \calR_{F_{\frakp}, E}(\delta_{\frakp, 1}') & & \calR_{F_{\frakp}, E}(\delta_{\frakp, 0}')
    \end{tikzcd},
\end{equation} where \[
    \delta_{\frakp, n-i}' = \delta_{\frakp, n-i}\left(\mathrm{Nm}^{F_{\frakp}}_{\Q_p}z\right)^m |\mathrm{Nm}^{F_{\frakp}}_{\Q_p}z|^m = \left(\prod_{\sigma: F_{\frakp} \hookrightarrow \overline{\Q}_p} \sigma(z)^{-k_{\frakp, \sigma, i}+m} \right)|\mathrm{Nm}^{F_{\frakp}}_{\Q_p}z|^{n-i}.
\] By taking cohomology, we have the connecting homomorphism \[
    \partial: H^0(\calR_{F_{\frakp}, E}(\delta_{\frakp, 0}')) \rightarrow H^1(\calR_{F_{\frakp}, E}(\delta_{\frakp, 1}')) = H_{\st}^1(\calR_{F_{\frakp}, E}(\delta_{\frakp, 1}')),
\] where the equation follows from Lemma \ref{Lemma: H1st = H1 for a special character}. Denoted by $\alpha_{\delta_{\frakp, 1}'}^*$ and $\beta_{\delta_{\frakp, 1}'}^*$ the two classes in $H^1(\calR_{F_{\frakp, E}}(\delta_{\frakp, 1}'))$ in Lemma \ref{Lemma: unqiue number defined by a semistable extension}, we know from \emph{loc. cit.} that $\partial$ gives rise to a unique number $\calL(W_{\frakp})\in E$ such that \[
    \beta_{\delta_{\frakp, 1}'}^* + \calL(W_{\frakp}) \alpha_{\delta_{\frakp, 1}'}^* \in \image \partial.
\]

We claim that \begin{equation}\label{eq: LGB as a product}
    \calL_{\GB}(\rho(m)) = \prod_{\frakp|p} \calL(W_{\frakp}).
\end{equation}
Note that, in the definition of $\calL_{\GB}(\rho(m))$, one studies the cohomology of $\Gr_1^{\GB}\D_{\rig}^{\dagger}(\rho_{\frakp}(m))$. However, we are now having cohomology classes in $H^1(\Gr_0^{\GB}\D_{\rig}^{\dagger}(\rho_{\frakp}(m)))$. To resolve this, we look at the short exact sequence \[
    \begin{tikzcd}[row sep = tiny, column sep = tiny]
        0 \arrow[r] & \left(\Gr_1^{\GB}\D_{\rig}(\rho_{\frakp}(m))\right)^{\vee}(\chi_{\cyc}) \arrow[r]\arrow[d, "\cong"] & W_{\frakp}^{\vee}(\chi_{\cyc}) \arrow[r] & \left(\Gr_0^{\GB}\D_{\rig}(\rho_{\frakp}(m))\right)^{\vee}(\chi_{\cyc}) \arrow[r] \arrow[d, "\cong"] & 0\\
        & \calR_{F_{\frakp}, E}(\kappa_{\frakp, 0}) && \calR_{F_{\frakp}, E}(\kappa_{\frakp, 1})
    \end{tikzcd}.
\] By \cite[Proposition 2.2.4]{Benois}, the Greenberg--Benois $\calL$-invariant computed by this exact sequence (at each $\frakp$) is the same as $\calL_{\GB}(\rho(m))$. Here, \[
    \kappa_{\frakp, i}(z) = \left(\prod_{\sigma: F_{\frakp} \hookrightarrow \overline{\Q}_p} \sigma(z)^{k_{\frakp, \sigma, n-i}-m+1} \right) |\mathrm{Nm}^{F_{\frakp}}_{\Q_p}z|^{1-i}
\] and we want to compute $\calL_{\GB}(\rho(m))$ using the cohomology of $\calR_{F_{\frakp}, E}(\kappa_{\frakp, 1})$.

By (B2), we have $k_{\frakp, \sigma, n-1}-m+1\leq 0$ and we let $u_{\frakp} := \min_{\sigma}\{ k_{\frakp, \sigma, n-1}-m+1\}$. By \cite[(2.8)]{Rosso}, there is an injection \[
    H^1(\calR_{\Q_p, E}(z^{u_{\frakp}})) \hookrightarrow H^1(\calR_{F_{\frakp}, E}(\kappa_{\frakp, 1})), \quad \begin{array}{l}
        x_{u_{\frakp}} \mapsto x_{k_{\frakp, \bullet, n-1}-m+1}   \\
        y_{u_{\frakp}} \mapsto y_{k_{\frakp, \bullet, n-1}-m+1}  
    \end{array}, 
\] where $x_{u_{\frakp}}$, $x_{k_{\frakp, \bullet, n-1}-m+1}$, $y_{u_{\frakp}}$, and $ y_{k_{\frakp, \bullet, n-1}-m+1}$ are as defined in \emph{loc. cit.}.\footnote{ For the convenience of the readers, we briefly recall the definitions of $x_{u_{\frakp}}$ and $y_{u_{\frakp}}$. The definitions for $x_{k_{\frakp, \bullet, n-1}-m+1}$ and $y_{k_{\frakp, \bullet, n-1}-m+1}$ are similar; we refer the readers to \cite[pp. 1233, 1234]{Rosso} for the precise definitions. Given $\alpha = (a,b)\in \calR_{\Q_p, E}(z^{u_{\frakp}})^{\oplus 2}$, one can define an extension $D_{\alpha}$ as in \eqref{eq: construction of an extension}, which defines a class $\mathrm{cl}(a,b)\in H^1(\calR_{\Q_p, E}(z^{u_{\frakp}}))$. We simplify the notation and write $e$ for the basis for $\calR_{\Q_p, E}(z^{u_{\frakp}})$, then $x_{u_{\frakp}} := \mathrm{cl}(t^{-u_{\frakp}}e, 0)$ and $y_{u_{\frakp}} = \log\chi_{\cyc}(\gamma)\mathrm{cl}(0, t^{-u_{\frakp}}e)$, where recall $\gamma$ is a (fixed) topological generator for $\Gamma = \Gal(\Q_p(\zeta_{p^{\infty}})/\Q_p)$ and $t=\log[\epsilon]\in B_{\dR}^+$. Here, $\epsilon\in \calO_{\C_p}^{\flat} = \varprojlim_{a\mapsto a^p}\calO_{\C_p}$ is defined by the fixed compatible system of primitive $p$-power roots of unity (see the beginning of \S \ref{subsection: (phi, Gamma)}) and we implicitly use the fact that certain subring of $\calR_{\Q_p}$ can be embedded into $B_{\dR}^+$ (see \cite[\S 1.2.2]{Benois}).} \footnote{ This injection comes from a natural injection $\calR_{\Q_p, E}(z^{u_{\frakp}}) \hookrightarrow \Ind^{\Gamma_{\Q_p}}_{\Gamma_{F_{\frakp}}} \calR_{F_{\frakp}, E}(\kappa_{\frakp, 1})$. By duality, we have a natural projection $\Ind^{\Gamma_{\Q_p}}_{\Gamma_{F_{\frakp}}}\calR_{F_{\frakp}, E}(\delta_{\frakp, 1}') \twoheadrightarrow \calR_{\Q_p, R}(z^{-u_{\frakp+1}}|z|)$ as well as $H^1(\calR_{F_{\frakp}, E}(\delta_{\frakp, 1}')) \twoheadrightarrow H^1(\calR_{\Q_p, E}(z^{-u_{\frakp}+1}|z|))$. According to \cite[Theorem 1.5.7]{Benois}, $x_{u_{\frakp}}$ (resp., $y_{u_{\frakp}}$) is dual to $\beta_{-u_{\frakp}+1}^*$ (resp., $-\alpha_{-u_{\frakp}+1}^*$). Thus, we may choose $v_{\delta_{\frakp, 1}'}\in \calD_{\st}(\calR_{F_{\frakp}, E}(\delta_{\frakp, 1}'))$ such that the corresponding $\alpha_{\delta_{\frakp, 1}'}^* = \cl(v_{\delta_{\frakp, 1}'}, 0,0) \mapsto \alpha_{-u_{\frakp}+1}^*$ and $\beta_{\delta_{\frakp, 1}'}^* = -\cl(0, 0, v_{\delta_{\frakp, 1}'})\mapsto \beta_{-u_{\frakp}+1}^*$.}
By the discussion on \cite[pp. 1238]{Rosso}, we have a commutative diagram \begin{equation}\label{eq: diagram computing GB L-inv explicitly}
    \begin{tikzcd}[column sep = tiny]
        & \scalemath{0.8}{ H^1\left(\bigoplus_{\frakp|p}D_{\frakp}^{\vee}(1-m), \Ind^{\Q}_{F}\rho^{\vee}(1-m)\right) }\arrow[rd, "\iota_c"]\arrow[ld, "\iota_f"']\arrow[d]\\
        \scalemath{0.8}{ \bigoplus_{\frakp|p} Ex_{k_{\frakp, \bullet, n-1}-m+1} } \arrow[r, hook] & \scalemath{0.8}{ \bigoplus_{\frakp|p} H^1(\calR_{F_{\frakp}, E}(\kappa_{\frakp, 1})) } & \scalemath{0.8}{ \bigoplus_{\frakp|p}Ey_{k_{\frakp, \bullet, n-1}-m+1} } \arrow[l, hook']
    \end{tikzcd}
\end{equation} where $\iota_c$ is an isomorphism. Moreover, [\emph{op. cit.}, Corollary 3.9] yields that \[
    \calL_{\GB}(\rho(m)) = \det(\iota_f \circ \iota_c^{-1}).
\]
In particular, if $\calL_{\frakp}\in E$ such that \begin{align*}
    \scalemath{0.8}{ \calL_{\frakp} x_{k_{\frakp, \bullet, n-1}-m+1}  + y_{k_{\frakp, \bullet, n-1}-m+1} \in \image \left(H^1\left(\bigoplus_{\frakp|p}D_{\frakp}^{\vee}(1-m), \Ind^{\Q_p}_{F}\rho^{\vee}(1-m)\right) \rightarrow H^1(\calR_{F_{\frakp}, E}(\kappa_{\frakp, 1})) \right) },
\end{align*} then \[
    \calL_{\GB}(\rho(m)) = \prod_{\frakp|p} \calL_{\frakp}. 
\]

By definition, $H^1\left(\bigoplus_{\frakp|p}D_{\frakp}^{\vee}(1-m), \Ind^{\Q}_{F}\rho^{\vee}(1-m)\right) \cong \bigoplus_{\frakp|p}\frac{H^1(W_{\frakp}^{\vee}(\chi_{\cyc}))}{H_f^1(W_{\frakp}^{\vee}(\chi_{\cyc}))}$. The vertical morphism in \eqref{eq: diagram computing GB L-inv explicitly} is compatible with the natural morphism \[
    H^1(W_{\frakp}^{\vee}(\chi_{\cyc})) \rightarrow H^1(\calR_{F_{\frakp}, E}(\kappa_{\frakp, 1})),
\] induced from the short exact sequence \eqref{eq: dual short exact sequence}. Note that the exact sequence \[
    H^1(W_{\frakp}^{\vee}(\chi_{\cyc})) \rightarrow H^1(\calR_{F_{\frakp}, E}(\kappa_{\frakp, 1})) \xrightarrow{\partial} H^2(\calR_{F_{\frakp}, E}(\kappa_{\frakp, 0}))
\] is dual to the exact sequence \[
    H^0(\calR_{F_{\frakp}, E}(\delta_{\frakp, 0}')) \xrightarrow{\partial} H^1(\calR_{F_{\frakp}, E}(\delta_{\frakp, 1}')) \rightarrow H^1(W_{\frakp}).
\]
We have $\partial(\calL_{\frakp} x_{k_{\frakp, \bullet, n-1}-m+1} + y_{k_{\frakp, \bullet, n-1}-m+1}) = 0 \in H^2(\calR_{F_{\frakp}, E}(\kappa_{\frakp, 0}))$ and $\beta_{\delta_{\frakp, 1}'}^* + \calL(W_{\frakp}) \alpha_{\delta_{\frakp, 1}'}^* = 0 \in H^1(W_{\frakp})$. Moreover, using the relation between $x_{k_{\frakp, \bullet, n-1}-m+1}$ (resp., $y_{k_{\frakp, \bullet, n-1}-m+1}$) and $\beta_{\delta_{\frakp, 1}'}^*$ (resp., $-\alpha_{\delta_{\frakp, 1}'}^*$), one sees that \[
    \calL_{\frakp} = \calL(W_{\frakp}),
\] which concludes our claim.

\vspace{2mm}

\noindent \textbf{Step 3. } Conclusion. 

By construction, $W_{\frakp}$ defines a class (see \eqref{eq: dual short exact sequence})\begin{align*}
    \cl(W_{\frakp}) & \in \mathrm{Ext}_{(\varphi, \Gamma)}^1(\calR_{F_{\frakp}, E}(\delta_{\frakp, 0}'), \calR_{F_{\frakp}, E}(\delta_{\frakp, 1}')) \\
    & \cong \mathrm{Ext}_{(\varphi, \Gamma)}^1(\calR_{F_{\frakp}, E}, \calR_{F_{\frakp}, E}(\delta_{\frakp, 1}'\delta_{\frakp, 0}'^{-1})) \\
    & \cong H^1(\calR_{F_{\frakp}, E}(\delta_{\frakp, 1}\delta_{\frakp, 0}^{-1})).
\end{align*} Note that, as classes in $H^1(\calR_{F_{\frakp}, E}(\delta_{\frakp, 1}\delta_{\frakp, 0}^{-1}))$, we have \[
    \cl(W_{\frakp}) = \cl(\widetilde{W}_{\frakp}).
\] Unwinding everything, we have \[
    c\cl(\calL(W_{\frakp})v_{\frakp, 1}, 0, -v_{\frakp, 1}) = \cl(W_{\frakp}) = \cl(\widetilde{W}_{\frakp})  \stackrel{\eqref{eq: main computation in Step 1}}{=} \cl(-\calL_{\FM}(\rho_{\frakp})v_{\frakp, 1}, 0, v_{\frakp, 1})
\] for some $c\in E$. In particular, we conclude that \[     
    \calL_{\FM}(\rho_{\frakp}) = -\calL(W_{\frakp})
\] and so, \[
    \calL_{\GB}(\rho(m)) = \prod_{\frakp|p} -\calL_{\FM}(\rho_{\frakp})
\]
by \eqref{eq: LGB as a product}.

\end{proof}

\printbibliography[heading=bibintoc]

\vspace{10mm}

\begin{tabular}{l}
    School of Mathematics and Statistics\\
    University College Dublin\\
    Belfield Dublin 4, Ireland\\
    \textit{E-mail address: }\texttt{ju-feng.wu@ucd.ie}
\end{tabular}

\end{document}